\documentclass[11pt,a4paper]{amsart}
\usepackage[T1]{fontenc}
\usepackage[utf8]{inputenc}
\usepackage{lmodern,amsmath,amssymb,amsthm,mathtools,microtype}
\usepackage[margin=25mm]{geometry}
\usepackage{booktabs,longtable,array,tabularx,enumitem,xurl,needspace}
\usepackage[hidelinks]{hyperref}
\hypersetup{pdftitle={A certified classical-strip comparison for the de Bruijn-Newman constant},
 pdfauthor={Michel Planat},
 pdfsubject={Classical-strip comparison below 0.148612 under explicit analytic premises}}
\newtheorem{theorem}{Theorem}[section]
\newtheorem{proposition}[theorem]{Proposition}
\newtheorem{lemma}[theorem]{Lemma}

\theoremstyle{definition}
\theoremstyle{remark}\newtheorem{remark}[theorem]{Remark}
\newcommand{\Rea}{\operatorname{Re}}
\newcommand{\Ima}{\operatorname{Im}}
\newcommand{\dd}{\,\mathrm d}
\newcommand{\eps}{\varepsilon}
\newcommand{\HH}{\mathsf H}
\newcommand{\ceil}[1]{\left\lceil#1\right\rceil}
\newcommand{\floor}[1]{\left\lfloor#1\right\rfloor}
\newcommand{\e}{\mathrm e}
\newcommand{\qstart}{q_{\mathrm{in}}}
\newcommand{\qend}{q_{\mathrm e}}
\setlist{itemsep=3pt,topsep=5pt,leftmargin=1.8em}
\title[A classical-strip comparison below \(0.148612\)]{A certified classical-strip comparison\\
for the de Bruijn--Newman constant}
\author{Michel Planat}
\address{Institut FEMTO-ST, CNRS UMR 6174, Universit\'e Marie et Louis Pasteur, Besan\c{c}on, France}
\date{13 October 2026}
\begin{document}
\begin{abstract}
We give a directed comparison beginning with the classical zero strip
\(\lvert\Ima z\rvert\le1\) for the de Bruijn--Newman heat deformation.
Under explicitly stated complete model-to-function source, localisation
and global contact premises, translated-probe estimates imply
\[
 \Lambda\le\frac{148611789321711}{10^{15}}
       =0.148611789321711<0.148612.
\]
The construction uses Raymond's finite source data and 57 fixed
translated kernels, independently checked by rational Bernstein
certificates. An explicit finite Weyl estimate improves the positive
first-head tail and its centred logarithmic moment. A centred cone
transfer retains the original physical error allowances, exceptional
debit and additive cone reserve, and charges the change in existing
phase-surplus credits. Directed tail enclosures are regenerated at
192 and 256 bits; a separate exact-rational auditor accepts all
19,023 comparison steps. The controlled decrease from the complete
two-source comparison is \(0.000021760452379\).
We also prove a compatible matched logarithmic tail transfer and
identify the finite requirements for stronger exponent pairs.
A supplementary section records Raymond's ceiling below \(0.116398\)
under the distinct \(7/8\) zero-free hypothesis.
\end{abstract}
\maketitle

\section{Introduction and the main result}

Use the Polymath normalisation
\begin{equation}\label{eq:heat}
 H_t(z)=\int_0^\infty \e^{tu^2}\Phi(u)\cos(zu)\dd u,\qquad
 \Phi(u)=\sum_{n\ge1}(2\pi^2n^4\e^{9u}-3\pi n^2\e^{5u})
                    \e^{-\pi n^2\e^{4u}}.
\end{equation}
Then \(H_t\) is real and even, \(H_0(z)=\xi((1+iz)/2)/8\), and
\(\partial_tH_t=-\partial_z^2H_t\).
The de Bruijn--Newman constant is characterised by
\[
 H_t\text{ has only real zeros}\quad\Longleftrightarrow\quad t\ge\Lambda.
\]
The underlying existence and strip-contraction results are due to
de Bruijn and Newman \cite{debruijn,newman}.
Rodgers and Tao proved \(\Lambda\ge0\) \cite{rodgerstao}.
Effective approximation and numerical barrier arguments were developed
by Polymath \cite{polymath}; Platt and Trudgian's finite-RH verification
also gives the published consequence \(\Lambda\le0.2\) \cite{pt}.

The critical strip alone gives \(\lvert\Ima z\rvert\le1\) for the zeros
of \(H_0\). The collective interaction of a highest nonreal zero with
the remaining zeros strengthens ordinary strip contraction. An off-zero
logarithmic derivative converts this interaction into a quantity
accessible to a finite approximation; this viewpoint was developed in
\cite{planat}. Raymond's translated-probe construction
\cite{raymondclassical,raymondfoundations,raymondretuned}
retains the common phases of two nearby observations and supplies
complete finite source bounds at prescribed heights.
The classical-strip ceiling reported in \cite{raymondclassical} is
\(0.149395073827688<0.149396\).

The present contribution is a new comparison from the classical strip
using both complete source schedules and the enlarged fixed-kernel
catalogue in \cite{raymondretuned}. The source estimates in that release
do not require its narrower initial zero strip: that additional input
enters the scalar comparison through the initial squared height.
This separation permits a new classical-strip computation.
We further strengthen its complete source by an explicit finite
Weyl tail estimate and a centred cone transfer. The unrefined
two-provider calculation remains a controlled comparison; the
tail-based decrease is established by a complete directed replay.

\Needspace{11\baselineskip}
\begin{theorem}[Classical-strip transfer]\label{thm:main}
Assume the true-function source, finite localisation and uniform
first-contact premises \textup{(A1)--(A3)} of
Section~\ref{sec:premises}. With the two complete source providers refined in
Section~\ref{sec:exponent}, and the 57 fixed kernels specified in
Appendix~\ref{app:kernels}, the directed comparison of
Section~\ref{sec:certificate} gives
\begin{equation}\label{eq:main}
 \boxed{\Lambda\le\frac{148611789321711}{10^{15}}
              =0.148611789321711<0.148612.}
\end{equation}
No hypothesis \(\HH_\alpha\), with \(\alpha<1\), concerning a fixed
zero-free half-plane for the Riemann zeta function is used.
\end{theorem}

The conclusion is an implication under explicit analytic premises.
The rational certificate proves the numerical comparison once those
premises are supplied; it does not prove its own interpretation as a
statement about \(H_t\). This distinction is part of the theorem, and
not merely a qualification of its numerical presentation.
The ceiling in \eqref{eq:main} is the smallest of the classical-strip
comparisons certified in this paper. A claim that it is an established
unconditional record would require validation of the complete analytic
assembly.

The paper has three complementary contributions. A complete source
portfolio is converted into a globally valid strip comparison with an
independently checked terminal ceiling. A matched logarithmic tail
transfer preserves the coefficient overlap that survives a cutoff.
Finally, a centred cone transfer combines an explicit finite Weyl
estimate with the retained phase credits and yields the numerical
gain in Section~\ref{sec:exponent}. These are alternative complete
source suppliers; their gains are not added to one another.

\section{Collective contraction and translated observations}
\label{sec:contraction}

For \(h>0\), define the conjugate-pair kernel
\begin{equation}\label{eq:K}
 K_h(d,b)=\frac{h-b}{d^2+(h-b)^2}
         +\frac{h+b}{d^2+(h+b)^2}.
\end{equation}
At a simple zero \(z_0=x+iy\), \(y>0\), whose imaginary height is
maximal among all zeros, write the other conjugate pairs as
\(a\pm ib\), \(0\le b\le y\). A real zero has half-pair weight, so
that it is counted once. The pair containing \(z_0\) is excluded from
the following sums.

\begin{lemma}[Exact highest-zero identity]\label{lem:dynamics}
At such a simple zero, with \(q=y^2\),
\begin{equation}\label{eq:dynamics}
 q'=-2-4y\sum_{\mathrm{other\ pairs}}K_y(x-a,b)
    =-2-4q\,G_t,\qquad
 G_t=\frac1y\sum_{\mathrm{other\ pairs}}K_y(x-a,b)\ge0.
\end{equation}
The imaginary parts of the paired zero sums converge absolutely.
\end{lemma}

\begin{proof}
Implicit differentiation of \(H_t(z_0(t))=0\) gives
\(z_0'=H_t''(z_0)/H_t'(z_0)\).
The canonical-product zero-motion identity gives twice the regular
zero sum at \(z_0\); any real constant in the logarithmic derivative
has no imaginary contribution. The own conjugate contributes
\(2/(z_0-\bar z_0)=-i/y\).
An additional conjugate pair contributes
\(-2K_y(x-a,b)\) to \(y'\), and a real zero contributes half
this expression with \(b=0\).
Multiplication by \(2y\) proves \eqref{eq:dynamics}.
Absolute convergence of the imaginary parts follows from the paired
product representation and the zero-counting estimate; the summands
at large real abscissa are \(O(y/|a|^2)\).
Every numerator in \eqref{eq:K} is nonnegative when \(b\le y\).
The zero-motion and product framework is the one in
\cite{polymath,raymondfoundations}.
\end{proof}

Above all zeros, put
\begin{equation}\label{eq:S}
 S_t(u,h)=-\Ima\frac{H_t'(u+ih)}{H_t(u+ih)},\qquad
 \bar S_{t,h,r}(x)=\frac{S_t(x-rh,h)+S_t(x+rh,h)}2.
\end{equation}
The paired product expresses \(S_t\) as the sum of \(K_h\), with the
same real-zero convention. Define
\[
 J_{h,r}(d,b)=\frac{K_h(d+rh,b)+K_h(d-rh,b)}2.
\]
A fixed kernel certificate is a rational triple
\((r,\rho,\lambda)\), \(0<\rho<1\), \(0<\lambda\le2\), such that
\begin{equation}\label{eq:kernelineq}
 2K_y(d,b)\ge\lambda\frac yh J_{h,r}(d,b)
 \quad(0<y\le\rho h,\ 0\le b\le y,\ d\in\mathbb R).
\end{equation}
The distinguished singular point \(d=0,b=y\) is excluded. That pair
is removed before the inequality is summed.

\subsection{The complete distinguished-pair subtraction}
The own pair contributes
\begin{equation}\label{eq:self}
 \frac1h\,s_r(q/h^2),\qquad
 s_r(w)=\frac{2(1+r^2-w)}
                  {(1+r^2-w)^2+4r^2w}
\end{equation}
to \(\bar S_{t,h,r}(x)\).
A rational cap \(M\) satisfying \(s_r(w)\le M\) for every
\(0\le w\le\rho^2\) is verified by nonnegativity of
\begin{equation}\label{eq:cap}
 M w^2+\{M(2r^2-2)+2\}w+
              M(1+r^2)^2-2(1+r^2).
\end{equation}
The endpoints and any interior minimum of this quadratic are checked
exactly. A cap at one contact height would not justify its use over a
whole comparison step.

\begin{proposition}[Complete source to contraction rate]\label{prop:rate}
Suppose \(\bar S_{t,h,r}(x)>\ell\) on the full contact domain and
\(q\le\rho^2h^2\). If the cap above is valid and
\[
 c=\frac{2\lambda}{h}\left(\ell-\frac Mh\right)>0,
\]
then at a simple highest-zero contact
\begin{equation}\label{eq:rate}
 q'<-2-cq.
\end{equation}
\end{proposition}

\begin{proof}
Sum \eqref{eq:kernelineq} over all pairs except the own pair.
Equations \eqref{eq:dynamics} and \eqref{eq:S} give
\[
 q'\le-2-\frac{2\lambda}{h}q
          \left(\bar S_{t,h,r}(x)-\frac{s_r(q/h^2)}h\right).
 \]
Use the complete source floor and the whole-range cap.
\end{proof}

\subsection{Rational certification of the fixed kernels}
\label{subsec:bernstein}
The 57 triples comprise Raymond's original 13 certificates and 44
additional certificates from \cite{raymondretuned}. We independently
check the entire catalogue.
After scaling \(h\) to one, put
\[
 D=d^2/h^2,\quad v=y^2/h^2,\quad B=b^2/h^2,\quad R=r^2,
\]
and define
\begin{align*}
 C&=D+R+1+B, & A&=C^2+4DR-4B,\\
 D_4&=A^2-16DRC^2, & G&=A(C-2B)-8DRC,\\
 Q_*&=(D+v-B)^2+4DB.
\end{align*}
Combining the two translations before clearing denominators gives
\begin{equation}\label{eq:poly}
 P=(D+v-B)D_4-\frac{\lambda}{2}Q_*G.
\end{equation}
The inequality \(P\ge0\) is equivalent to
\eqref{eq:kernelineq}. Indeed,
\(K_y/y=2(D+v-B)/Q_*\) and \(J_{1,r}=2G/D_4\).
The translated denominators are positive since \(b/h<1\);
\(Q_*\) vanishes only at the excluded singular point.

Set \(D=z/(1-z)\), \(v=\rho^2p\), \(B=\rho^2pu\). The polynomial
\[
 \widetilde P(z,p,u)=(1-z)^5
       P\bigl(z/(1-z),\rho^2p,\rho^2pu\bigr)
\]
has degree at most five in each variable. Its required domain is the
closed cube \([0,1]^3\); \(z=1\) is included by polynomial continuation.
For a power-coefficient vector \(a\) of degree at most five, the
one-variable Bernstein coefficients are
\[
 b_k=\sum_{j=0}^k a_j\,\frac{\binom kj}{\binom5j}.
\]
Apply this conversion along each coordinate. If all tensor Bernstein
coefficients on a box are nonnegative, the polynomial is nonnegative
there. Half-box restrictions are obtained from the exact de Casteljau
formulas
\[
 b_k^{L}=2^{-k}\sum_{j=0}^k\binom kj b_j,\qquad
 b_k^{R}=2^{-(5-k)}\sum_{j=k}^5\binom{5-k}{j-k}b_j.
\]

The independent implementation reconstructs \eqref{eq:poly} by rational
polynomial arithmetic. At each unresolved box it splits into both
children, and accepts a leaf only when every Bernstein coefficient is
nonnegative. Integer rescaling uses positive factors only.
The resulting covering partitions have 434 leaves in total:
41 for the original 13 kernels and 393 for the new 44.
Every quadratic cap \eqref{eq:cap} also passes on its full declared
interval. No stored polynomial, subdivision tree or Raymond verifier is
imported by this implementation.

\section{Complete source providers and analytic premises}
\label{sec:premises}

The original numerical inputs are the complete providers
\texttt{SOURCE\_OLD\_ENHANCED.json} and
\texttt{SOURCE\_RETUNED.json} from the retuning evidence
\cite{raymondretuned}. They are stored under
\texttt{inputs/original/} in the certificate package.
Section~\ref{sec:exponent} derives augmented complete providers,
stored under \texttt{inputs/refined/} with the same filenames.
The observation schedules and kernel catalogue are identical.
Each provider has 608 closed time cells
\begin{equation}\label{eq:cells}
 I_i=[i/4000,(i+1)/4000],\qquad 0\le i<608,
\end{equation}
covering \([0,19/125]\).
Each cell carries its own rational observation height
\(h_{p,i}\), with provider index \(p\in\{0,1\}\), and a complete source
floor \(\ell^{\mathrm{old}}_{p,i,r}\) for every retained shift.
The refined floor satisfies
\(\ell^{\mathrm W}_{p,i,r}\ge\ell^{\mathrm{old}}_{p,i,r}\).
In the main comparison, \(\ell_{p,i,r}\) denotes the accepted
refined floor; the superscripts distinguish the controlled comparisons.
The older schedule has heights at most \(109/50\); the retuned
schedule has heights in \([69/100,47/25]\).
The comparison consumes a floor only with its owning provider, height
and time cell.

Let
\[
 X=6000000185827,\qquad N_0=690988,\qquad
 L=\log(x/(4\pi)),\qquad s=L/2-\log N_0.
\]
The six frequency classes are
\[
 [0,1/20],\ [1/20,3/20],\ [3/20,2/5],\
 [2/5,1],\ [1,2],\ [2,\infty).
\]
They cover the physical range \(x\ge X-2\).
For centres \(x\ge X\) and displacements \(|rh|<2\), the source
construction treats all eleven same or adjacent class pairs
\((j,j)\), \(0\le j\le5\), and \((j,j+1)\), \(0\le j<5\).
The stored common floor is the minimum over these complete pair cases.
The final unbounded class includes every future native cutoff.

\subsection{The true-function source correspondence}
The finite source construction uses
\[
 b_t(n)=\e^{t(\log n)^2/4},\qquad
 E_t(s)=\prod_{p\in\{2,3,5\}}(1-b_t(p)p^{-s}),
\]
and the physical leading exponent
\[
 s_U^0=(1+h)/2+tL/4-i(x/2+\pi t/8).
\]
Both finite branches and the complete reflected and product-tail
labels are retained. Its normalised finite model has the form
\begin{equation}\label{eq:labelled}
 Q=1+\sum_{j>0}a_jz_j,\qquad
 V=\sum_{j>0}\omega_j a_jz_j,\qquad
 Q_h=-V/2,
\end{equation}
with \(a_j,\omega_j\ge0\), \(|z_j|=1\).
Denominator exclusion must precede division.
The physical transfer has the form
\begin{equation}\label{eq:physical}
 S_t(x,h)>P-\tfrac12\Rea(V/Q)+\Delta_E,\qquad
 P=\ell_{\mathrm{base}}+C/2,\quad \Delta_E\ge0,
\end{equation}
where the base floor has already paid the complete physical value,
derivative and edge-pair errors. The Euler phase deficit
\(\Delta_E\) and all selected finite-coefficient deficits are treated
jointly at the translated observations.

The analytic derivation is supplied in
\cite{raymondfoundations,raymondmatched,raymondretuned}.
It extends the approximation to observation heights above one.
Its Cauchy estimate concerns the holomorphic error on a frozen-cutoff
disk, including entering pairs and transition allowances.
Applying a restricted-height approximation or differentiating a
changing truncation would not establish these estimates.

\subsection{Premises of Theorem~\ref{thm:main}}
\label{subsec:ledger}
We state the remaining interfaces in the precise form consumed by the
comparison.
\begin{enumerate}[label=\textbf{(A\arabic*)}]
\item\label{prem:source}
\emph{Complete model-to-function source correspondence.}
For each original provider, cell and shift, the corrected approximation,
the labelled correspondence \eqref{eq:labelled}, denominator exclusion
and physical transfer \eqref{eq:physical} hold uniformly on the full
physical domains. The complete amplitude and moment enclosures,
exceptional debits, allocated paired phase credits and physical error
allowances of the original source assembly are valid there. In particular,
\[
 \bar S_{t,h_{p,i},r}(x)>\ell^{\mathrm{old}}_{p,i,r}
       \quad(t\in I_i,\ x\ge X).
\]
The assertion includes both translated observations, all eleven class
pairs, retained reflected and product-tail labels, and every future
activation in the unbounded class. Every quantity is used at its owning
provider, height and time cell. These complete interfaces, together with
the finite enclosures and transfer proved in Section~\ref{sec:exponent},
imply the refined source inequalities for \(\ell^{\mathrm W}_{p,i,r}\).
Only the earlier tabulated lower bounds, without their complete
model-to-function correspondence and charged allowances, would not
license that refinement.
\item\label{prem:wall}
\emph{Finite localisation.}
The retained finite-head wall is valid for the true function on
\[
 x=\pm X,\qquad |\Ima z|\le1,\qquad 0\le t\le41/250,
\]
and the finite-RH input and homotopy argument exclude nonreal zeros
from \(|\Rea z|\le X\) throughout this time range.
The wall must use the corrected broad remainder allowance described
below.
\item\label{prem:contact}
\emph{Uniform global contact reduction.}
On each compact positive-time interval used in the comparison, the
positive-time real-simple tail, confinement and local zero-motion
framework supply compactness for a prospective first contact.
The backward Hermite splitting rule applies at a multiple zero.
Together these statements justify the first-contact argument for the
piecewise smooth positive strip boundary with upward restarts.
\end{enumerate}
Premise (A1) is the complete true-function source correspondence,
not merely a property of a JSON file or an assumption that the
refined numerical floors already hold. Premises (A2)--(A3) refer to the written analytic
localisation and contact arguments, not to a PASS marker.
The supplied evidence supports these premises, but their complete
mathematical validation has not been independently completed in this
audit.

\subsection{Finite-RH clearance and the corrected wall}
Theorem 1 of Platt--Trudgian \cite{pt} verifies RH through
\(T_*=3000175332800\). The required ordinate is \(X/2\), with the
stronger exact clearance
\[
 T_*-(X+1)/2=175239886>0.
\]
The finite-RH computation is imported, not rerun.
For the retained wall the ideal floor is \(79/100\); the ideal
displacement is below \(1/100\).
The corrected broad remainder and the remaining physical allowances
are bounded above by \(1/180\), giving
\begin{equation}\label{eq:wall}
 |H_t/B_t|>79/100-1/100-1/180
             =697/900>3/4.
\end{equation}
Here \(B_t\) is the nonzero physical normaliser of the supplied
approximation. The older narrow wall-error field is not a premise.

For clarity, the corrected Gaussian comparison in
\cite[Appendix B]{raymondretuned} uses \(a=11/12\) rather than
the historical \(5/6\).
Writing \(D=T-3\), its complete Gaussian expectation is
\begin{equation}\label{eq:gaussian}
 \int_{\mathbb R}\exp\!\left(\frac{u^2+a}{D}+cu\right)
        \frac{\e^{-(u-\sigma)^2/t}}{\sqrt{\pi t}}\dd u
 =
 \frac{\exp\!\left\{a/D+
       \dfrac{\sigma^2/D+\sigma c+tc^2/4}{1-t/D}\right\}}
      {\sqrt{1-t/D}},\quad D>t>0.
\end{equation}
Completing the square proves the identity; its \(t=0\) value is
obtained by continuity.
The source domains include \(x\ge X-4\), \(0.54\le y\le2.41\)
on the Cauchy disks and \(-1\le\sigma\le2\) for the all-real-abscissa
remainder.
The safety checks verify the associated rational constants and the
complete error-row allowances. They do not mechanically establish
the imported all-real-abscissa expansion.

\section{The directed classical-strip certificate}
\label{sec:certificate}

Fix the strict initial squared-height enclosure and terminal target
\begin{equation}\label{eq:parameters}
 q_0=\left(\frac{1000001}{1000000}\right)^2,\qquad
 \qend=10^{-8},\qquad
 \Delta t=\frac1{128000}.
\end{equation}
Thus each source cell is divided into 32 substeps.
At the beginning of a substep with incoming rational state \(\qstart\),
a candidate consists of one complete provider and one fixed kernel.
It is eligible only if
\begin{equation}\label{eq:eligible}
 \qstart\le\rho^2h^2,\quad |rh|<2,\quad
 c=\frac{2\lambda}{h}\left(\ell-\frac Mh\right)>0.
\end{equation}
The producer chooses the largest eligible rational \(c\), with a
deterministic ordering of ties. This is a finite local selection
rule, not a proof of global optimality.

The older provider is essential at the classical starting height.
The largest catalogue range parameter is \(19/40\), so the retuned
initial observation height \(47/25\) permits at most
\[
 (47/25)^2(19/40)^2=797449/10^6<1<q_0.
\]
It therefore supplies no eligible initial kernel for this seed.
The first accepted record uses the older height \(109/50\) and
the triple \((r,\rho,\lambda)=(19/40,23/50,2)\).
Its full rate is \(302869653560843/29702500000000\).
As the comparison height decreases, the retuned provider becomes
eligible and is selected when its complete rate is larger.

With the selected constant rate, the exact comparison solution is
\begin{equation}\label{eq:solution}
 \mathcal F_c(q,u)=\left(q+\frac2c\right)\e^{-cu}-\frac2c.
\end{equation}
For \(0<x<1\), the even Taylor polynomial
\[
 E_{18}(x)=\sum_{k=0}^{18}\frac{(-x)^k}{k!}
\]
is an upper bound for \(\e^{-x}\), by the alternating-series
remainder.
The producer restarts upward at
\begin{equation}\label{eq:restart}
 q_{\mathrm{next}}=
 2^{-96}\ceil{2^{96}
     \left\{\left(\qstart+\frac2c\right)E_{18}(c\Delta t)
                         -\frac2c\right\}}.
\end{equation}
Every nonterminal restart obeys
\(\qend<q_{\mathrm{next}}<\qstart\).
The kernel therefore remains valid throughout the decreasing
exact segment.

\subsection{Strict separation near time zero}
Take \(\delta_0=10^{-8}\).
For the first selected rate, the exact check is
\begin{equation}\label{eq:initial}
 q_0-(2+c_0q_0)\delta_0>1,\qquad 0<\delta_0<\Delta t.
\end{equation}
Since \(\mathcal F_{c_0}(q_0,t)\) is decreasing while positive and
its derivative is at least \(-2-c_0q_0\), it remains above one on
\([0,\delta_0]\). Classical strip control thus gives strict
separation there. Uniform positive-time contact control is used
after \(\delta_0\); no positive-time theorem is extrapolated back
to \(t=0\).

\subsection{The terminal step}
When a full substep would reach \(\qend\), the exact time required
within that step is
\[
 \tau_*=\frac1c\log(1+z),\qquad
 z=\frac{c(\qstart-\qend)}{2+c\qend}.
\]
The producer uses the odd degree-11 polynomial
\begin{equation}\label{eq:terminal}
 \tau_{\mathrm{up}}=
 \frac1c\sum_{k=1}^{11}\frac{(-1)^{k+1}z^k}{k},
 \qquad 0<z<1,\quad 0<\tau_{\mathrm{up}}\le\Delta t.
\end{equation}
It is an upper bound for \(\tau_*\).
At the terminal time the strip has squared half-height at most
\(\qend\). De Bruijn contraction then spends the full additional
\(\qend/2\) to make all zeros real.
For a terminal substep beginning at \(t_n=n\Delta t\), the final
ceiling is
\begin{equation}\label{eq:ceiling}
 U=10^{-15}\ceil{10^{15}(t_n+\tau_{\mathrm{up}}+\qend/2)}.
\end{equation}
The terminal source interval and the wall time range both contain
the required times.

\begin{proof}[Proof of Theorem~\ref{thm:main}]
Proposition~\ref{prop:refinedsources} converts the complete
model-to-function premise (A1) and the finite tail certificate into
the refined source inequalities consumed below.
Construct the piecewise exact boundary \(\sqrt{\mathcal F_c(q,u)}\),
with upward restarts \eqref{eq:restart}.
It is initially strictly above all zeros by \eqref{eq:initial}.
Suppose a first positive contact occurs on a decreasing segment.
Premise (A3) localises the contact to a finite abscissa.
Premise (A2) and symmetry place an upper-right contact at \(x>X\).

At a multiple nonreal contact, backward Hermite splitting produces
an upward displacement of order \(\sqrt{\eps}\) at time \(t-\eps\).
The positive smooth comparison boundary changes by \(O(\eps)\).
This contradicts the definition of first contact.
At a simple contact, \eqref{eq:eligible} ensures that the entire
kernel domain applies. Premise (A1) and
Proposition~\ref{prop:rate} give \(q'<-2-cq\), whereas the
comparison segment has derivative \(-2-cq\).
A zero branch reaching the boundary from below has the opposite
one-sided derivative requirement, a contradiction.
At a knot, the left segment suffices and the upward restart preserves
separation.

Thus all zeros remain below the comparison strip until its squared
height reaches \(\qend\). Equations \eqref{eq:terminal}--\eqref{eq:ceiling}
and de Bruijn contraction give \(\Lambda\le U\).
The separately audited history evaluates
\(U=148611789321711/10^{15}\), as detailed below.
\end{proof}

\subsection{Independent numerical verification}
The producer imports no Raymond comparison, selector or
elementary-function module.
The separate auditor imports neither the producer nor its selector.
All substantive comparisons are performed with Python's exact
integer and rational arithmetic.
For the exponential enclosure it uses degree 20:
\[
 E_{20}(x)=\sum_{k=0}^{20}\frac{(-x)^k}{k!}
             \ge\e^{-x}.
\]
It requires every recorded restart to be at least the corresponding
degree-20 enclosure, as well as to lie on the declared dyadic grid.
For the terminal logarithm it independently sets \(w=z/(2+z)\) and
uses
\begin{equation}\label{eq:auditorlog}
 \log(1+z)\le
 2\sum_{k=0}^{6}\frac{w^{2k+1}}{2k+1}
       +\frac{2w^{15}}{15(1-w^2)}.
\end{equation}
The positive-tail bound follows by replacing all omitted denominators
by 15.

The accepted main history has 19,023 records.
Provider 0 supplies 8,512 steps and provider 1 supplies 10,511.
The last record is substep 14 of cell 594, beginning at
\[
 t_n=19022/128000=0.148609375.
\]
Its incoming squared height is approximately
\(4.8287726951\times10^{-6}\), and its rational terminal lag is
approximately \(2.4093217100\times10^{-6}\).
These two approximations are descriptive only; the history stores
and checks their full rational values.
The rates used range approximately from \(10.1968\) to \(22.1778\).

The auditor checks the source identities, all closed time cells,
continuous record ordering, source-provider ownership, observation
heights, frequency-pair coverage, kernel eligibility, complete
own-pair caps, positive rates, initial separation, all upward
restarts, the terminal lag and the full ordinary strip cost.
Nine corrupted histories are rejected, including an omitted own-pair
charge, an inflated source floor, a changed height, a time gap and
a downward restart.
These negative controls test the certificate consumer; they are not
additional analytic premises.

\section{Controlled comparisons and the verification boundary}
\label{sec:comparisons}

The following comparisons share the classical seed
\eqref{eq:parameters}, the same 32 substeps per source cell, the same
global own-pair cap policy and fixed shifts only.
Their histories pass the same independent auditor.
\begin{table}[ht]
\centering\small
\begin{tabularx}{\textwidth}{@{}Xr@{}}
\toprule
Available complete estimates & Directed ceiling\\
\midrule
Older matched source; original 13 kernels
 & \(0.149056792554129\)\\
Same source; all 57 kernels
 & \(0.148639484094797\)\\
Both complete original providers; all 57 kernels
 & \(0.148633549774090\)\\
Both complete providers with finite Weyl refinement; all 57 kernels
 & \(0.148611789321711\)\\
\bottomrule
\end{tabularx}
\caption{Controlled classical-strip comparisons. Every row has the
same analytic premises (A1)--(A3).}\label{tab:controlled}
\end{table}
Within this sequence, the enlarged kernel catalogue accounts for
\(0.000417308459332\), and access to the second complete provider
adds \(0.000005934320707\).
The finite Weyl source refinement then contributes the controlled
decrease \(0.000021760452379\).
The total reduction across these four complete computations is
\(0.000445003232418\).
These are differences between complete computations, not independent
source contributions that may be added to another calculation.
The final ceiling is also \(0.000783284505977\) below the classical
ceiling reported in \cite{raymondclassical}; that comparison changes
more than one component of the construction.

No estimate is interpolated between observation heights.
When switching providers, the whole source estimate is switched:
its amplitude and moment bounds, denominator, physical errors,
class-pair minima and height are never assembled from different
providers. No continuum-shift estimate is used.
These restrictions make the finite certificate directly reviewable.

\begin{table}[ht]
\centering\small
\begin{tabularx}{\textwidth}{@{}p{0.34\textwidth}X@{}}
\toprule
Layer & Verification established here\\
\midrule
Collective and translated algebra &
Written derivations in Section~\ref{sec:contraction}.\\
Full kernel catalogue &
Independent rational polynomial proof of 57 kernels and full caps.\\
Scalar ceiling &
19,023 exact steps accepted by a separately written auditor.\\
Original source reconstruction &
Stored interval evidence reassembled and checked; original
transcendental production not repeated.\\
Finite Weyl tail refinement &
Fresh directed tail and mass production at 192 and 256 bits;
complete rational cone, credit-loss and case-minimum checks.\\
Finite-head wall and finite RH &
Retained wall certificates and corrected allowances checked;
original production and finite-RH computation not repeated.\\
Complete true-function assembly &
Explicit premises (A1)--(A3); independent end-to-end analytic
validation remains open in this audit.\\
\bottomrule
\end{tabularx}
\caption{Scope of certification. Code independence at the rational
layer does not establish the shared physical interpretation.}
\label{tab:scope}
\end{table}

The full stored-source replay of the supplied retuning archive,
including regeneration of all eight supplied comparison histories,
completed successfully. The source reconstruction checks
127,680 native acceptance inequalities, 1,284,096 joint conditions,
eight complete production partitions and 26,752 cross-construction
interval overlaps. It reproduces the selected source and 274 changed
height rows; the counts of 9,120 positive head-pair credits and 6,384
matched-logarithmic improvements agree with the deposit.

These original-source checks reconstruct the supplied interval
enclosures and their arithmetic consequences. The separate tail
refinement in Section~\ref{sec:exponent} freshly regenerates its own
transcendental weights and masses at two precisions. Both sets of native
constructions use Arb/FLINT; the geometry, cone assembly and classical
comparison consumers use exact rational arithmetic.
Neither the finite-RH computation nor the external Lean build discussed
in Appendix~\ref{app:qrh} was executed.

\section{A matched logarithmic centred-tail transfer}
\label{sec:tail}

Finite logarithmic matching and oscillatory cancellation address the
same source quotient. To combine them, one must retain the matching
that survives the chosen tail cutoff and pay for the removed terms
in both amplitude and logarithmic moment.
We first recall the positive finite-coefficient structure from
\cite{raymondmatched,raymondretuned}.

\subsection{Formal coefficients and finite matching}
Let \(*\) denote Dirichlet convolution, \(\delta\) its identity and
\(\mathcal Df(n)=(\log n)f(n)\).
All logarithms and exponentials of coefficient sequences below are
formal: each coefficient uses finitely many proper factors.
Put
\begin{equation}\label{eq:coefficients}
 c_t=b_t*\!\!\prod_{p\in\{2,3,5\}}^{*}
            \bigl(\delta-b_t(p)\delta_p\bigr),\qquad
 \widetilde d_t=\mathcal D\log_*c_t .
\end{equation}
For \(t>0\), the heat-weighted infinite numerical Dirichlet series
does not converge at any fixed abscissa. None is evaluated here.

\begin{lemma}[Positive finite coefficients]\label{lem:positive}
For \(t\ge0\), \(c_t(n)\) and \(\widetilde d_t(n)\) are nonnegative,
as are their time derivatives.
Moreover \(c_t(1)=1\), and
\(c_t(p)=\widetilde d_t(p)=0\) for \(p=2,3,5\).
\end{lemma}

\begin{proof}
Write \(d_t=\mathcal D\log_*b_t\). The derivation rule for
\(\mathcal D\), together with \(\partial_tb_t=\mathcal D^2b_t/4\),
gives
\[
 \partial_td_t(n)=\frac{\log^2n}{4}d_t(n)
       +\frac{\log n}{4}
          \sum_{\substack{ab=n\\a,b\ge2}}d_t(a)d_t(b).
\]
At \(t=0\), \(d_0\) is the nonnegative von Mangoldt sequence.
Induction on \(n\), followed by the integrating-factor formula,
proves positivity, including all time derivatives.
At a selected prime power \(p^k\), put
\(\kappa_p=(\log p)^2/4\) and
\(e_k=d_t(p^k)-(\log p)\e^{k\kappa_pt}\).
For \(k\ge2\),
\[
 e_k'=k^2\kappa_pe_k+
       (k^2-k)\kappa_p(\log p)\e^{k\kappa_pt}
       +\frac{k\log p}{4}
                 \sum_{j=1}^{k-1}d_t(p^j)d_t(p^{k-j}),\quad e_k(0)=0.
\]
The forcing and all its derivatives are nonnegative.
For \(k=1\), the residual is identically zero.
Elsewhere \(\widetilde d_t=d_t\).
Finally,
\[
 (\log n)c_t(n)=
       \sum_{\substack{d\mid n\\d\ge2}}\widetilde d_t(d)c_t(n/d)
\]
proves the claims for \(c_t\) by induction.
\end{proof}

Retain the positive first head through an integer \(J\), together
with all other licensed reflected and product-tail labels, in
\((Q_0,V_0)\). At a fixed physical source point write
\begin{equation}\label{eq:decomposition}
 Q=Q_0+R,\qquad V=V_0+\nu R+E,\qquad
 |R|\le\eta,\quad |E|\le\eps,\quad \nu\ge0.
\end{equation}
The notation \(\eta\) distinguishes the tail allowance from the
kernel range parameter \(\rho\).
Let \(\bar A_0,\bar B_0\) bound the nonconstant amplitude and
weighted moment of the retained core. Choose \(K\le J\) and put
\begin{align}
 T&=\sum_{2\le d\le K}\widetilde d_t(d)d^{-\sigma}z_d,
 &D_K&=\sum_{2\le d\le K}\widetilde d_t(d)d^{-\sigma},
 \label{eq:logsup}\\
 Z_J&=\sum_{2\le d\le K}\widetilde d_t(d)d^{-\sigma}
                  \sum_{2\le m\le J/d}c_t(m)m^{-\sigma}.
 \label{eq:overlap}
\end{align}
Only the first-head phases satisfy the required relation
\(z_dz_m=z_{dm}\); all other labels remain separate.
The shared mass in \(Q_0T\) and \(V_0\), through this retained head,
is \(D_K+Z_J\).
Lemma~\ref{lem:positive} licenses its removal from both labelled
sums before the triangle inequality, giving
\begin{equation}\label{eq:matching}
 |Q_0T-V_0|
 \le\bar B_0+(1+\bar A_0)D_K-2(D_K+Z_J)
 =\bar B_0-(1-\bar A_0)D_K-2Z_J.
\end{equation}
The right side bounds the remaining labelled mass and is nonnegative.
The constant label is responsible for \(D_K\); the overlap \(Z_J\)
excludes that constant.

\begin{proposition}[Matched centred-tail transfer]\label{prop:tail}
Assume \eqref{eq:decomposition}--\eqref{eq:matching} and
\(g=1-\bar A_0-\eta>0\). Then
\begin{equation}\label{eq:tailbound}
 \boxed{\Rea\frac VQ\le
       \frac{\bar B_0+\nu\eta+\eps-2Z_J}{1-\bar A_0-\eta}
              -(D_K-\Rea T).}
\end{equation}
The bounds must hold uniformly on the full physical rectangle where
this source will be consumed.
\end{proposition}

\begin{proof}
Use the exact identity
\[
 QT-V=(Q_0T-V_0)+(T-\nu)R-E.
\]
Since \(|T|\le D_K\) and \(|Q|\ge g\), a nonnegative residual
budget is
\[
 U_0=\bar B_0-(1-\bar A_0)D_K-2Z_J+(D_K+\nu)\eta+\eps.
\]
Thus \(\Rea(V/Q)\le\Rea T+U_0/g\).
The cancellation
\[
 D_K+U_0/g=\frac{\bar B_0+\nu\eta+\eps-2Z_J}{g}
\]
gives \eqref{eq:tailbound}.
\end{proof}

Keeping the original complete physical error allowance in
\eqref{eq:physical}, define
\[
 F_J=P-\frac{\bar B_0+\nu\eta+\eps-2Z_{J,\mathrm{low}}}{2g}.
\]
Then
\[
 S_t(x,h)>F_J+\tfrac12(D_K-\Rea T)+\Delta_E .
\]
The same common-phase paired deficit argument can be applied at the
two translated observations, using coefficients bounded from below
at both points and separate denominators.
This gives a compatible complete source estimate rather than an
addition of incompatible gains.

\begin{remark}[Net improvement criterion]\label{rem:net}
At a fixed physical height, compare the complete quotient
\[
 \frac{\bar B_0+\nu\eta+\eps-2Z_{J,\mathrm{low}}}
      {1-\bar A_0-\eta}
 \quad\text{with}\quad
 \frac{\bar B-2Z_{\mathrm{full,low}}}{1-\bar A},
\]
with the physical errors and paired phase deficits fully accounted
for in each candidate. Removing a tail loses some overlap credit.
Smaller \(\eta,\eps\) alone therefore do not prove a larger source floor.
A complete old estimate can remain a fallback. Pointwise maxima
of valid complete candidates are taken before minima over the
source-class pairs.
\end{remark}

\section{Explicit finite Weyl tails and the numerical improvement}
\label{sec:exponent}

We now use oscillatory cancellation to improve a complete source, rather
than subtracting a zeta growth exponent from \(\Lambda\). The matched
transfer in Section~\ref{sec:tail} and the centred cone below are
alternative suppliers. The numerical gain uses the cone supplier,
which retains the allocated finite phase labels and explicitly pays
for the change in their credits.

\subsection{A centred cone preserving the retained phase credits}

At a physical source point, write the complete labelled model as
\[
 Q=1+\sum a_jz_j,\qquad V=\sum\omega_ja_jz_j,
 \qquad a_j,\omega_j\ge0,\quad |z_j|=1.
\]
The nonconstant amplitude and weighted moment have complete bounds
\(A^-\le\sum a_j\le A^+\), \(\sum\omega_ja_j\le B^+\).
Let \(C,D\) be the old moment-ratio bound and exceptional debit,
with
\begin{equation}\label{eq:oldcone}
 C(1+A^-)^2-(1+A^-)B^+-D>0.
\end{equation}
The scalar source has paid its physical perturbation using these
complete amplitude and moment bounds.

Remove only the positive first-head terms \(J<n\le M=690988\).
The retained head contains every allocated constant/head label
through 1024 and every allocated head/head label through 2048.
The reflected, product-tail and activation labels stay in the core.
Write
\[
 Q=Q_0+R,\qquad V=V_0+\nu R+E,
 \qquad |R|\le\eta,\quad |E|\le\eps,
\]
and let \(A_0,B_0\) be the actual retained amplitude and moment.

\begin{lemma}[Centred cone transfer]\label{lem:cone}
Assume \(0<C'\le C\) and \(\nu\ge2C'\). The complete cone at
the new parameter \(C'\)
is bounded below by
\begin{equation}\label{eq:newcone}
 C'(1+A_0+\eta)^2
  -(1+A_0+\eta)(B_0+\nu\eta+\eps)-D
  +\mathcal P_{\rm retained}(C'),
\end{equation}
where \(\mathcal P_{\rm retained}(C')\) is the sum of the licensed,
nonnegative retained phase surpluses. The old debit \(D\) remains valid
if \(C'\le C\).
\end{lemma}

\begin{proof}
For the core, expand \(C'|Q_0|^2-\Rea(V_0\overline{Q_0})\)
about the aligned phases. Its constant/head surplus is
\((\omega_j-2C')a_j(1-\Rea z_j)\); its unordered head/head surplus is
\((\omega_j+\omega_k-2C')a_ja_k(1-\Rea(z_j\overline{z_k}))\).
For a negative frequency coefficient \(-b\), charging \(2b\) times
its label mass leaves the nonnegative surplus \(b(1+\cos\theta)\)
times that mass. For a positive coefficient, the surplus is
\(b(1-\cos\theta)\). This treats both retained chiralities.
Decreasing \(C\) decreases every exceptional magnitude, so the old
debit covers the new exceptional charges with nonnegative spare margin.
Omitted nonnegative surpluses may be discarded.

The difference between the full cone and the core cone is
\[
 \Rea\bigl((2C'-\nu)R\overline{Q_0}-V_0\overline R\bigr)
 +(C'-\nu)|R|^2-\Rea\bigl(E(\overline{Q_0}+\overline R)\bigr).
\]
It is at least
\[
 -\bigl((\nu-2C')(1+A_0)+B_0\bigr)\eta
 -(\nu-C')\eta^2-\eps(1+A_0+\eta).
\]
Combining this with the core aligned-phase baseline gives
\eqref{eq:newcone} by direct expansion.
\end{proof}

Suppose the removed positive tail has mass between \(a_-\) and
\(a_+\), and moment at least \(b_-\). Put
\begin{align*}
 A_{0,-}&=\max(0,A^--a_+),& A_{0,+}&=A^+-a_-,\\
 B_{0,+}&=B^+-b_-,&
 S_*&=1+A_{0,-}+\eta,\\
 A_*&=A_{0,+}+\eta,& B_*&=B_{0,+}+\nu\eta+\eps.
\end{align*}
Define the reduction
\begin{equation}\label{eq:delta}
 \Delta=\frac{B^+}{1+A^-}+\frac D{(1+A^-)^2}
          -\frac{B_*}{S_*}-\frac D{S_*^2},\qquad C'=C-\Delta.
\end{equation}
The consumer requires
\[
 0\le A_{0,-}\le A_{0,+},\qquad B_{0,+}\ge0,\qquad
 \eta,\eps\ge0,
\]
\[
 \eta\le a_-,\quad B_*\le B^+,\quad
 0<C'\le C,\quad \nu\ge2C',\quad 2C'S_*>B_*.
\]
The last condition licenses replacing actual \(1+A_0+\eta\)
by \(S_*\) in the quadratic cone lower bound.
Equation~\eqref{eq:delta} preserves exactly the original additive
reserve:
\[
 C'-B_*/S_*-D/S_*^2
 =C-B^+/(1+A^-)-D/(1+A^-)^2>0.
\]
Since \(A_*\le A^+\) and \(B_*\le B^+\), the original complete
physical perturbation charge remains conservative.
With an empty removed tail, \(\Delta=0\); the formula supplies no
gain merely by reoptimising the old cone reserve.

\subsubsection{Charging the change in surplus credits}

The existing paired credits use distances of frequencies or frequency
sums from \([2C_{\min},2C_{\max}]\).
Changing one class bound by at most \(\Delta\) reduces each distance
by at most \(2\Delta\). The retained normalization is
\(2(1+A)^2\). For \(A_{\max}<1\), the total normalized
constant/head mass is at most
\(A_{\max}/[2(1+A_{\max})^2]\): choose either side of the
coupled coefficient minimum and use the increasing function
\(A/(1+A)^2\). The disjoint unordered head/head mass is at most
\(A_{\max}^2/2\).
Thus a valid loss allowance is
\begin{equation}\label{eq:creditloss}
 L_{jk}=\Delta\,
     \frac{A_{\max}+A_{\max}^2/2}{(1+A_{\max})^2}.
\end{equation}
All allocated labels survive the cutoff. The Euler credit is unchanged.

Only class 0 is refined. Its paired source cases \((0,0)\) and
\((0,1)\) receive respectively
\(\Delta/2-L_{00}\) and \(\Delta/4-L_{01}\).
The other nine cases keep their complete old bounds.
The consumer takes the minimum over all eleven cases, rounds down
on the \(10^{-12}\) lattice, and takes the maximum with the complete
old source floor. This is the step which turns local cancellation
into a licensed global source improvement.
The matched logarithmic alternative remains available on earlier
cells; no overlap credit from that alternative is added here.

\subsection{Finite Fourier grouping and weighted prefixes}

For \(v=\log n\), expand the finite convolution
\eqref{eq:coefficients} as
\[
 c_t(n)=\e^{tv^2/4}
       \sum_{\substack{d\mid30\\d\mid n}}
          \kappa_d(t)\e^{-tv\log d/2},\qquad
 \kappa_d(t)=\mu(d)\exp\!\left\{
       \frac t4\left((\log d)^2+\sum_{p\mid d}(\log p)^2\right)\right\}.
\]
The finite identity
\(\mathbf1_{d\mid n}=d^{-1}\sum_{a=0}^{d-1}\e^{2\pi ian/d}\)
therefore gives the exact grouping
\begin{equation}\label{eq:fourier}
 c_t(n)=\e^{tv^2/4}\sum_{j=0}^{29}F_j(v;t)\e^{2\pi ijn/30},
 \quad
 F_j(v;t)=
 \sum_{\substack{d\mid30\\30/d\mid j}}
        \frac{\kappa_d(t)}d\,\e^{-tv\log d/2}.
\end{equation}
Grouping precedes every absolute value.
For physical phase frequency \(\Omega=x/2+\pi t/8\), the thirty
twisted phases are
\[
 f_j(u)=\frac{\Omega}{2\pi}\log u+\frac{ju}{30}.
\]
All derivatives of order at least two are independent of \(j\).
Thus one uniform derivative-test envelope can cover all twists.

\begin{proposition}[Uniform weighted-prefix transfer]\label{prop:prefix}
Suppose \(\mathcal B(u)\) bounds every twisted prefix
\(\bigl|\sum_{J<n\le u}\e^{2\pi if_j(n)}\bigr|\), uniformly
on the full source rectangle and for every \(0\le j<30\).
For continuously differentiable weights \(W_j\),
\begin{equation}\label{eq:prefix}
 \left|\sum_{J<n\le M}W_j(n)\e^{2\pi if_j(n)}\right|
 \le |W_j(M)|\mathcal B(M)
       +\int_J^M\mathcal B(u)|W_j'(u)|\dd u .
\end{equation}
Apply this to
\[
 W_j(u)=\e^{t(\log u)^2/4-\sigma\log u}F_j(\log u;t),
 \qquad (\log u-\nu)W_j(u),
 \quad \sigma=(1+h)/2+tL/4.
\]
The sum of the thirty resulting bounds supplies
\(\eta\) and \(\eps\) in \eqref{eq:decomposition} for the removed
first-head tail.
\end{proposition}
\begin{proof}
Discrete partial summation gives the endpoint term minus the
integral against the prefix step function. Taking absolute values
proves \eqref{eq:prefix}; the centred weight proves the moment
assertion. The coefficient identity \eqref{eq:fourier} is exact.
\end{proof}

\subsection{An explicit finite Weyl envelope}

For the first frequency class \(s\in[0,1/20]\), the complete endpoint
allowance is covered by
\[
 L=2\log M+s',\qquad -10^{-9}\le s'\le1/10+10^{-9},
 \qquad \alpha=\frac{\Omega}{2\pi}=\e^L+t/16.
\]
Here \(M=N_0=690988\) is the fixed first-head limit; the letter
\(M\) in the contraction-rate formulas denotes an own-pair cap.
The tail estimate concerns only this fixed positive first head.
On an integer parent block \((a,b]\), let
\(\alpha_-\le\alpha\le\alpha_+\), and set
\(l=2\alpha_-/b^3\), \(u=2\alpha_+/a^3\).
The phase difference \(f_j(v+r)-f_j(v)\) has positive, decreasing
second derivative between \(lr\) and \(ur\).
Yang's explicit second-derivative test~\cite[Lemma 2.1]{yang}
therefore bounds each nonempty correlation by
\[
 c_2\frac u{\sqrt l}N\sqrt r+uNr+
     \frac{c_2}{\sqrt l\sqrt r},\qquad
 c_2=113/50>4/\sqrt\pi,
\]
where using the parent prefix length \(N\), rather than \(N-r\),
only enlarges the bound.

For every integer differencing parameter \(q\ge1\), zero padding
and Cauchy--Schwarz give
\begin{equation}\label{eq:finiteweyl}
 \left|\sum_{n=1}^N z_n\right|^2
 \le(N-1+q)\left\{\frac Nq+
       \frac2q\sum_{r=1}^{q-1}(1-r/q)|\mathcal C_r|\right\},
\end{equation}
where \(z_n=\e^{2\pi if_j(a+n)}\) on \(1\le n\le N\),
zero elsewhere, and
\(\mathcal C_r=\sum_{n=1}^{(N-r)_+}z_{n+r}\overline{z_n}\).
Indeed, summing the \(q\)-term moving sums over their \(N+q-1\)
possible starting positions counts each \(z_n\) exactly \(q\) times.
Their squared norms sum to
\(qN+2\sum_{r=1}^{q-1}(q-r)\Rea\mathcal C_r\).
Cauchy--Schwarz and division by \(q^2\) prove
\eqref{eq:finiteweyl}, including \(q>N\).
Using
\[
 \sum_{r=1}^{q-1}(1-r/q)\sqrt r\le\tfrac4{15}q^{3/2},\quad
 \sum_{r=1}^{q-1}(1-r/q)r\le\tfrac16q^2,\quad
 \sum_{r=1}^{q-1}\frac{1-r/q}{\sqrt r}\le\tfrac43\sqrt q,
\]
one obtains the explicit prefix bound
\begin{equation}\label{eq:Weyl}
 \mathcal B_q(N)^2=(N-1+q)
 \left(\frac Nq+\frac{8c_2u}{15\sqrt l}N\sqrt q
                +\frac u3Nq+\frac{8c_2}{3\sqrt l\sqrt q}\right).
\end{equation}
The first weighted-sum inequality follows by integrating the concave
function \((1-r/q)\sqrt r\); the second follows from the exact sum
\(\sum_{r=1}^{q-1}(1-r/q)r=(q^2-1)/6\);
the third follows by integrating a decreasing function.
All lower-order terms are retained. No condition \(q\le N\) is
needed after zero padding. The trivial bound \(N\) may be used as well.

\subsection{Uniform panels and positive removed-tail masses}

Each block is treated by Proposition~\ref{prop:prefix} after resetting
the prefix origin to its left integer endpoint. A fixed-parameter
envelope in \eqref{eq:Weyl} increases with prefix length. For each full
integration panel, its value at the right endpoint bounds the prefix
step function throughout that panel. Directed boxes enclose all weights
and their derivatives over the complete time cell and frequency subband.
The thirty twist bounds are added, block bounds are added, and subband
maxima are taken. No physical oscillatory sum is sampled.

The positive removed-tail masses are bounded independently by the
finite signed divisor convolution
\[
 \sum_{J<n\le M}c_t(n)n^{-\sigma}
 =\sum_{d\mid30}\mu(d)\e^{t\sum_{p\mid d}\log^2p/4-\sigma\log d}
       \sum_{J/d<m\le M/d}\e^{t\log^2m/4-\sigma\log m}.
\]
The log moment inserts \(\log m+\log d\).
Discrete midpoint bins enclose the raw positive summands with complete
second-derivative remainder bounds; all eight signed terms are combined
before extracting endpoints. For \(f(m)=\e^{t\log^2m/4-\sigma\log m}\)
and \(r=t\log m/2-\sigma\),
\begin{align*}
 f''(m)&=\frac{f(m)}{m^2}(r^2-r+t/2),\\
 ((\log m+\log d)f(m))''
 &=\frac{f(m)}{m^2}
     ((\log m+\log d)(r^2-r+t/2)+2r-1).
\end{align*}
Both convexity signs and upper caps are verified on each full integer
bin \([a,b]\). If \(N=b-a+1\) and \(c=(a+b)/2\),
\(\sum(m-c)=0\), \(\sum(m-c)^2=N(N^2-1)/12\).
For either raw summand \(g\), a verified bound \(0\le g''\le K_g\)
therefore gives
\[
 Ng(c)\le\sum_{m=a}^b g(m)
      \le Ng(c)+K_gN(N^2-1)/24.
\]
This is a complete Taylor remainder enclosure.
Positivity and time monotonicity of \(c_t(n)\), proved in
Lemma~\ref{lem:positive}, license the separated time--exponent lower and upper models.
This step bounds removed mass without replacing it by raw absolute
divisor masses.

\subsection{Complete source assembly}
For each cell and shift let \(b_{jk}\) be the original complete paired
floor for each of the eleven cases. Refine only class 0 using
\eqref{eq:delta}; all other class bounds remain unchanged. Define
\[
 b^{\mathrm W}_{00}=b_{00}+\Delta/2-L_{00},\qquad
 b^{\mathrm W}_{01}=b_{01}+\Delta/4-L_{01},
\]
and set \(b^{\mathrm W}_{jk}=b_{jk}\) for the other nine cases.
The accepted complete source is
\begin{equation}\label{eq:refinedfloor}
 \ell^{\mathrm W}_{p,i,r}=
 \max\left\{\ell^{\mathrm{old}}_{p,i,r},\,
      10^{-12}\floor{10^{12}\min_{(j,k)}b^{\mathrm W}_{jk}}\right\}.
\end{equation}
Every selected label lies below the retained cutoff. The Euler credits,
reflection, product-tail terms, activation allowances and physical errors
are retained. No logarithmic overlap credit from Proposition~\ref{prop:tail}
is added to this cone supplier.

\begin{proposition}[Complete refined providers]\label{prop:refinedsources}
Under the complete model-to-function correspondence in (A1), the
uniform tail enclosures and rational checks described here give
\[
 \bar S_{t,h_{p,i},r}(x)>\ell^{\mathrm W}_{p,i,r}
       \quad(t\in I_i,\ x\ge X)
\]
for every provider, cell and declared fixed shift.
The schedules, observation heights, kernel domains and original
physical error allowances are unchanged.
\end{proposition}
\begin{proof}
Lemma~\ref{lem:cone}, the positive cone reserve in \eqref{eq:delta},
and the complete bounds \(A_*\le A^+\), \(B_*\le B^+\) license the
individual source increase \(\Delta/2\) with the original physical
perturbation charge. At a pair of translated observations this gives
\(\Delta/2\) in case \((0,0)\) and \(\Delta/4\) in case \((0,1)\).
The loss allowance \eqref{eq:creditloss} pays for both retained
surplus chiralities. The other nine case bounds are complete original
bounds. Their full minimum therefore applies at every physical centre,
including transitions between adjacent classes. Downward rounding
preserves the lower bound. The maximum with the original complete
floor is legitimate because both candidates hold on the entire domain.
Outside the refined cells the original complete estimate applies,
including its unbounded class and all future activations.
\end{proof}
\subsection{Directed enclosures and the numerical gain}

The frozen cutoffs are selected from
\(131072,196608,262144,327680\).
The centre is \(\nu=10^{-12}\lfloor10^{12}\log J\rfloor\).
Each parent block has 64 integration panels, and the first frequency
class is split into two closed subbands. The block ratios are chosen
from \(11/10,5/4,3/2,2\); every chosen ratio is checked as a complete
integer partition. Floating computations only choose these hints
and an integer differencing parameter. They do not determine accepted
enclosure endpoints or source floors.

The mass bins use division 512. FLINT/Arb regenerates the logarithms,
exponentials, Fourier weights and mass bounds independently at
192 and 256 bits. Acceptance uses the weaker lower or upper bound
from the two runs and checks interval overlap.
This is two precision runs of one directed library, not two independent
transcendental libraries.
All phase-envelope square roots and all final source substitutions
are checked by a separate rational consumer.

\begin{center}
\begin{tabular}{lrrr}
\toprule
Cell & \(J\) & \(\eta\) & \(\eps\)\\
\midrule
400 & 262144 & 0.003131230 & 0.001542391 \\
480 & 196608 & 0.004442683 & 0.002702791 \\
560 & 196608 & 0.004214458 & 0.002607289 \\
594 & 131072 & 0.005300291 & 0.004086663 \\
\bottomrule
\end{tabular}
\end{center}

The logarithmic centre is rounded downward and independently
checked with a rational atanh enclosure for \(\log J\).
The complete source consumer accepts \(552\) provider--cell
refinements on 276 time cells and validates
\(6528\) rational prefix envelopes per precision run.
Every refined source retains its own prescribed height and the entire
57-kernel catalogue. Old complete floors remain fallbacks everywhere.
The scalar comparison is regenerated from time zero with
32 steps per source cell, exact exponential upper bounds, upward
dyadic restarts and the full terminal strip cost.

An independent rational auditor, importing no producer or selector,
accepts every one of the \(19023\) steps and all time joins,
kernel ranges, own-pair caps and source ownership decisions.
It also rejects nine deliberately damaged histories.
The controlled result is
\begin{center}
\begin{tabular}{lr}
\toprule
Complete classical-strip calculation & Directed ceiling\\
\midrule
Original complete two-source comparison & 0.148633549774090\\
Explicit finite Weyl refinement & 0.148611789321711\\
Exact decrease & 0.000021760452379\\
\bottomrule
\end{tabular}
\end{center}
The associated replay retains the same starting strip, kernels,
own-pair caps and physical allowances. Its change is therefore a
numerical consequence of the finite tail method, with the modified
surplus credits fully charged. It is not a gain from changing the
initial zero strip or from releasing an old rounding reserve.

\subsection{What a stronger exponent pair would change}
For the classical theory of exponent pairs and their application to exponential sums, see Graham and Kolesnik~\cite{grahamkolesnik}.

In the fixed-ratio regime \(J/M=\theta\in(0,1)\),
\(\Omega\asymp M^2\), \(L=2\log M+\delta\) with bounded
\(\delta\ge0\), and \(t,h\) in fixed compact sets with \(t\ge0\),
a uniform bound of exponent-pair shape
\[
 \mathcal B\ll(\Omega/M)^\kappa M^\lambda
            =O(M^{\kappa+\lambda})
\]
gives, by Proposition~\ref{prop:prefix},
\begin{equation}\label{eq:scaling}
 \eta,\eps\ll
 M^{\kappa+\lambda-(1+h)/2}
            \e^{-t(\log M)^2/4},
\end{equation}
up to any stipulated exponent loss.
To see the scale, write \(\log u=\log M+w\),
\(\log\theta\le w\le0\).
The heat exponent becomes
\[
 -\frac t4(\log M)^2+\frac t4w^2
       -\frac{t\delta}4(\log M+w).
\]
The remaining weight and its total variation have the indicated
power of \(M\), uniformly on this fixed logarithmic interval.
Choose \(\nu=\log J+O_\theta(1)\); the centred logarithm is bounded.

For this regime the relevant exponent-pair objective is
\(\kappa+\lambda\).
The Weyl pair \((1/6,2/3)\) has sum \(5/6\).
A pair with sum \(5/6-\delta_{\mathrm{pair}}\) improves the leading
power by \(M^{-\delta_{\mathrm{pair}}}\).
If its leading constant is \(C_{\mathrm{new}}\), a necessary
leading-term break-even comparison is
\[
 C_{\mathrm{new}}/C_{\mathrm{Weyl}}<M^{\delta_{\mathrm{pair}}}.
\]
At the actual finite scale, lower-order terms, uniformity, denominator
reserve, lost matching credit and physical allowances must also pass.
An asymptotic pair with an unspecified implied constant does not
certify a new finite source.

The ceiling in \eqref{eq:main} already uses the explicit finite Weyl
envelope \eqref{eq:Weyl} and the centred cone supplier. The complete
matched logarithmic estimates are retained on the earlier cells where
they were already licensed. A stronger explicit exponent pair may
improve the tail further, but its finite constants and lower-order
terms must satisfy the same cone, surplus-loss and complete-case checks.
The matched transfer of Section~\ref{sec:tail} remains a separate
complete alternative; no incompatible supplier gains are summed.

\section{Reproducibility, provenance and remaining work}
\label{sec:repro}

The ancillary certificate package contains the original and refined complete
source providers, all four histories of Table~\ref{tab:controlled},
the directed tail enclosures at both precisions, exact reassembled
parent budgets and eleven-case tables, frozen cutoff hints, a native
uniform-tail producer, separate rational source and scalar consumers,
the full 57-kernel checker and verification receipts.
The manifest identifies bytes; the numerical consumers establish the
finite inequalities. Original native source evidence and analytic
foundations retain their attribution in the cited Raymond release.

The original scalar, physical and kernel inputs retain Raymond's
provenance. The collective method retains its analytic provenance
\cite{planat}. The matched and centred cone transfers are proved in
Sections~\ref{sec:tail}--\ref{sec:exponent}; the new finite Weyl supplier
uses the explicit second-derivative estimate of \cite{yang} with all
finite terms retained. The source reconstruction reassembles the
original interval evidence, whereas the new tail weights and positive
mass bounds were freshly regenerated at 192 and 256 bits.

The standard-library verifier reassembles the complete refined source,
checks its rational phase envelopes and cone budgets, verifies all
four scalar histories and the full kernel catalogue, and exercises
finite algebra fixtures and nine damaged-history rejection controls.
Its replay mode regenerates all four scalar histories and compares
their compressed bytes. Native tail regeneration is a separate
documented operation requiring a compatible FLINT/Arb library and
NumPy. Two precision runs use one directed library; agreement is not
claimed as independence of transcendental implementations.

The main outstanding analytic task remains full validation of the
model-to-function, localisation and global-contact interfaces
(A1)--(A3), including their seams and enlarged-height domains.
The numerical tail refinement is now a complete finite gain under
those premises. Further numerical gains require complete uniform
suppliers and another directed replay; validation of the analytic
interfaces has a different role, completing the interpretation of
the ceiling as a true-function theorem.

\appendix
\section{Supplementary comparisons under the \texorpdfstring{\(7/8\)}{7/8} hypothesis}
\label{app:qrh}

For \(1/2<\alpha<1\), write \(\HH_\alpha\) for the assertion that
every nontrivial zeta zero has real part at most \(\alpha\).
Zeros on the boundary are permitted. Functional-equation symmetry
then gives \(1-\alpha\le\Rea s\le\alpha\).
Since \(H_0(z)=\xi((1+iz)/2)/8\),
\begin{equation}\label{eq:stripalpha}
 \HH_\alpha\quad\Longrightarrow\quad
            |\Ima z|\le2\alpha-1
              \quad\text{for every zero of }H_0.
\end{equation}
In particular, \(\HH_{7/8}\) supplies squared half-height \(9/16\).

Raymond's new main comparison \cite{raymondretuned} uses the strict
seed \(562501/10^6=9/16+10^{-6}\), the retuned physical source
and fixed kernels. Under its physical and localisation interfaces,
its principal implication is
\[
 \HH_{7/8}\ \Longrightarrow\
 \Lambda\le\frac{116397726408291}{10^{15}}
                =0.116397726408291<0.116398.
\]
The alternate exact history gives
\(14549715801039/125000000000000
 =0.116397726408312\).
Both have 14,899 steps; the stored-source audit with history
regeneration reproduced their compressed files byte-for-byte.

The external \(7/8\) statement is associated with OpenAI's
30 September 2026 quasi-Riemann-hypothesis release
\cite{openaiqrh}. The official formalization scope note
\cite{openaiscope} states that the \(7/8\) nonvanishing result is
formalized for the Riemann zeta function, and also describes broader
Dirichlet and Hecke statements. The latter are not needed here.
The supplied comparator configuration names the theorem
\texttt{OAI.riemannZeta\_ne\_zero\_of\_seven\_eighths\_lt\_re}
in \texttt{OAI.NumberTheory.DirichletL.Nonvanishing}.
An official scope declaration is distinct from executing the pinned
Lean build; no such build was run for this paper.

Accordingly, we retain \(\HH_{7/8}\) as an explicit external
mathematical input to this supplementary implication.
Validation of that input would remove this particular hypothesis
from the implication, while the physical-source and global-contact
assembly would still need its own validation.
The classical-strip theorem \ref{thm:main} has no dependence on
this external zero-free result.
The difference between \(0.148612\) and \(0.116398\) therefore includes
a stronger initial strip; it is not a comparison of arithmetic
accuracy.

\subsection{Hakimoglu-Brown's derivative comparison}
Cetin Hakimoglu-Brown \cite{hakimoglu} obtains the conditional ceiling
\(\Lambda\le13963/100000=0.13963\) under the same hypothesis
\(\HH_{7/8}\). His construction uses the published Polymath boundary
calculation to retain a finite-region strip of half-width \(1/5\),
rather than importing Raymond's supplied localisation construction.
At observation height \(p=61/20\), bounds for both
\(S(x,p,t)=-\Ima(H_t'/H_t)(x+ip)\) and its vertical derivative
support a strengthened pairwise inequality of the form
\[
 K_y(d,v)\ge y\left\{\frac{\lambda}{p}K_p(d,v)
                         -\mu\partial_pK_p(d,v)\right\}.
\]
A 24-step comparison then reaches squared height \(1/25\) at
\(t=0.11963\); the final classical strip contraction costs \(1/50\).
His finite scalar and polynomial certificates are specified in the
paper. Their specification is distinct from independent reproduction;
we have not independently replayed his full certificate or validated
all of its analytic estimates.

Raymond's reported ceiling \(0.116397726408291\) is numerically smaller
under the same initial half-plane hypothesis, but uses retuned
translated probes, matched finite logarithmic coefficients and
separate supplied analytic and numerical inputs. Hakimoglu-Brown's
derivative comparison therefore has a distinct mechanism and
dependency structure. Neither conditional ceiling is ranked against
Theorem~\ref{thm:main} by decimal size alone: the latter starts from
the classical strip and retains its explicit premises (A1)--(A3).

\subsection{A translated derivative pilot and its boundary}
The derivative mechanism also admits a translated pilot. Put
\[
 \begin{aligned} T_h(d,v)&=\tfrac12\{K_h(d-rh,v)+K_h(d+rh,v)\},\\
 D_h(d,v)&=\left.\tfrac12\partial_p
       \{K_p(d-rh,v)+K_p(d+rh,v)\}\right|_{p=h}.\end{aligned}
\]
The physical displacement \(rh\) is held fixed in this derivative.
For \(r=19/40\), \(V=211/1000\),
\(\lambda=10019/10000\) and \(\mu=1/500\), exact rational
Bernstein subdivision certifies
\[
 \frac{K_y(d,v)}{y}\ge\frac{\lambda}{h}T_h(d,v)-\mu D_h(d,v),
 \qquad 0<y^2/h^2\le V,\quad0\le v\le y,
\]
away from the singular distinguished pair. The sufficient lower
envelope is \(2/(d^2+Vh^2+3v^2)\). After multiplication by positive
pair denominators, scaling \(h=1\), and compactifying
\((d/h)^2=U/(1-U)\), \((v/h)^2=VW\), its numerator gives a
polynomial of degree \((8,8)\). Ten rational Bernstein leaves prove
positivity. The ancillary pilot code specifies this construction;
it has not been audited by an independent implementation.

For the distinguished pair, the existing cap
\(hT_h(0,y)\le M=333069367/200000000\) is retained, and a
separate rational Bernstein check gives \(h^2D_h(0,y)\ge-11/10\).
If simultaneous physical bounds \(\overline S\ge\ell\) and
\(\overline{\partial_pS}\le G\) were available, the resulting rate
would be
\[
 c=\frac4h\left\{\lambda(\ell-M/h)
                      -\mu\left(hG+\frac{11}{10h}\right)\right\}.
\]
On the initial classical cell, using the unchanged source floor and
\(h=109/50\), this exceeds the existing rate precisely when
\[
 G<\frac{5204523417656017}{2376200000000000}
       =2.190271617564185\ldots.
\]
The required uniform physical derivative bound has not been
certified. Existing pointwise source floors cannot be differentiated
to obtain it. Second-derivative and remainder estimates, followed by
a full directed descent, are needed before this pilot can supply a
new ceiling. It contributes no numerical gain to Theorem~\ref{thm:main}.

\section{Exact fixed-kernel catalogue}\label{app:kernels}

The original thirteen triples and their complete own-pair caps
are listed below. All displayed finite decimals are exact; rational
fractions are retained where needed.
The first column is the zero-based kernel index used in the histories.
\begin{center}\small
\begin{tabular}{@{}rrrrr@{}}\toprule
Index & \(r\) & \(\rho\) & \(\lambda\) & \(M\)\\\midrule
0 & \(0.5\) & \(0.45\) & \(2\) & \(1.618033989\)\\
1 & \(0.525\) & \(0.45\) & \(1.99\) & \(1.575653048\)\\
2 & \(0.5\) & \(0.46\) & \(1.97\) & \(1.618033989\)\\
3 & \(0.45\) & \(0.475\) & \(1.9\) & \(1.718428456\)\\
4 & \(0.575\) & \(0.35\) & \(1.99\) & \(1.503067131\)\\
5 & \(0.6\) & \(0.2\) & \(1.99\) & \(\frac{25}{17}\)\\
6 & \(0.5\) & \(0.45\) & \(1.99\) & \(1.618033989\)\\
7 & \(0.6\) & \(0.35\) & \(1.97\) & \(\frac{25}{17}\)\\
8 & \(0.475\) & \(0.46\) & \(2\) & \(1.665346835\)\\
9 & \(0.51\) & \(0.45\) & \(2\) & \(1.600531138\)\\
10 & \(0.49\) & \(0.455\) & \(2\) & \(1.636324261\)\\
11 & \(0.535\) & \(0.35\) & \(2\) & \(1.559923928\)\\
12 & \(0.565\) & \(0.2\) & \(2\) & \(1.516438236\)\\
\bottomrule\end{tabular}
\end{center}

The additional 44 triples have shifts \(r=23/40\) and \(r=3/5\).
For every row in the following table, the corresponding
\(\lambda\) is the exact displayed terminating decimal.
Their whole-range caps are
\[
 M_{23/40}=\frac{1503067131}{10^9},\qquad M_{3/5}=\frac{25}{17}.
\]
They are Raymond's constants \cite{raymondretuned}, independently
checked here by Section~\ref{subsec:bernstein}.
\begin{center}\small
\begin{tabular}{@{}rrr@{\qquad}rrr@{}}\toprule
\(\rho\) & \(\lambda_{23/40}\) & \(\lambda_{3/5}\) & \(\rho\) & \(\lambda_{23/40}\) & \(\lambda_{3/5}\)\\\midrule
\(0.05\) & \(1.99998\) & \(1.99819\) & \(0.35\) & \(1.99603\) & \(1.98841\)\\
\(0.1\) & \(1.99997\) & \(1.99784\) & \(0.36\) & \(1.99551\) & \(1.98751\)\\
\(0.15\) & \(1.99992\) & \(1.99719\) & \(0.37\) & \(1.99493\) & \(1.98654\)\\
\(0.2\) & \(1.99968\) & \(1.99612\) & \(0.38\) & \(1.99429\) & \(1.98549\)\\
\(0.25\) & \(1.99911\) & \(1.99448\) & \(0.39\) & \(1.99359\) & \(1.98436\)\\
\(0.28\) & \(1.99853\) & \(1.99312\) & \(0.4\) & \(1.99282\) & \(1.98313\)\\
\(0.3\) & \(1.998\) & \(1.99202\) & \(0.41\) & \(1.99197\) & \(1.98179\)\\
\(0.31\) & \(1.99768\) & \(1.9914\) & \(0.42\) & \(1.99105\) & \(1.98033\)\\
\(0.32\) & \(1.99733\) & \(1.99074\) & \(0.43\) & \(1.99003\) & \(1.97036\)\\
\(0.33\) & \(1.99694\) & \(1.99002\) & \(0.44\) & \(1.96906\) & \(1.93243\)\\
\(0.34\) & \(1.99651\) & \(1.98925\) & \(0.45\) & \(1.92508\) & \(1.89014\)\\
\bottomrule\end{tabular}
\end{center}

\section{Certificate identities and reproduction instructions}
\label{app:identities}

The principal fixed identities are:
\begin{center}\small
\begin{tabularx}{\textwidth}{@{}p{0.32\textwidth}X@{}}
\toprule
Item & SHA-256\\
\midrule
Older refined provider &
\nolinkurl{f047655ac110ab34767aa0db61cd3802888158ca5fa6b6ce68c33d22b2d22960}\\
Retuned refined provider &
\nolinkurl{0ee511e46d6125450e63b932b3369a44995ce386e24ef139611ed3b58e3e15eb}\\
Main classical history &
\nolinkurl{90d120c378589ca7f7e68508a4037c2a9e1b3245f403386dad3d36da9dc215b0}\\
\bottomrule
\end{tabularx}
\end{center}
After unpacking the ancillary certificate package, the standard-library command \texttt{python verify.py} checks the
integrity manifest and runs the independent rational source, geometry,
history and finite-fixture consumers. Adding \texttt{--regenerate}
also rebuilds all four classical histories in a separate temporary
directory and compares their compressed bytes with the supplied files.
The original provider identities and both directed tail-lane identities
are retained in the package manifest and audit receipts.
Neither standard-library mode reruns native transcendental production,
repeats finite RH, proves (A1)--(A3), or builds Lean.

For native tail regeneration, \texttt{code/setup\_runtime.py} extracts
the unchanged, hash-checked compatible runtime from the supplied
\texttt{Proof\_and\_Replay.zip}; NumPy is also required. Alternatively,
the documented environment variable selects an installed compatible
FLINT library. Running \texttt{code/directed\_tail.py} separately with
\texttt{--bits 192} and \texttt{--bits 256} regenerates all declared
tail rectangles. The README documents the full commands, dependencies
and subsequent exact source-consumer invocation.

\section*{Acknowledgements}
The author acknowledges Dale James Raymond's supplied manuscript,
source evidence and analytic foundations.
Their use and independent numerical rechecking do not imply
Raymond's endorsement of this new classical-strip transfer.

\section*{AI assistance}
OpenAI ChatGPT assisted in the derivations, manuscript preparation,
code reimplementation and recorded arithmetic checks.
AI assistance, agreement of programs and successful replay do not
constitute external mathematical refereeing.

\end{document}